\documentclass[11pt]{article}

\usepackage[a4paper,margin=2.5cm]{geometry}

\usepackage{amsmath}
\usepackage{amssymb}
\usepackage{amsthm}
\usepackage{hyperref}
\usepackage{graphicx}
\usepackage[table]{xcolor}
\usepackage{caption}
\usepackage{subcaption}
\usepackage{tikz}
\usepackage{mathtools}
\usepackage{booktabs}
\usepackage{pdflscape}
\usepackage{authblk}

\makeatletter
\newcommand\hw@temp{\phi}
\let\phi\varphi
\let\varphi\hw@temp

\renewcommand\hw@temp{\epsilon}
\let\epsilon\varepsilon
\let\varepsilon\hw@temp

\makeatother

\DeclarePairedDelimiterX{\set}[1]\lbrace\rbrace{\setaux #1||\endsetaux}
\def\setaux#1|#2|#3\endsetaux{\if\relax\detokenize{#2}\relax #1 \else #1 \;\delimsize\vert\; #2 \fi}

\newcommand\at[1]{#1\textsuperscript{@}}
\newcommand\nothing{--}
\newcommand\match[2]{\ensuremath{#2\mathbin{@}#1}}

\newcommand{\whitesq}{\ensuremath{\square}}
\newcommand{\blacksq}{\ensuremath{\blacksquare}}

\makeatletter
\protected\def\ccell#1#{%
    \kern-\fboxsep
    \@ccell{#1}%
}
\def\@ccell#1#2#3{%
    \colorbox#1{#2}{#3}%
    \kern-\fboxsep
}
\makeatother

\definecolor{baseorange}{HTML}{D55E00}
\definecolor{baseblue}{HTML}{0072B2}
\definecolor{basegreen}{HTML}{009E73}
\definecolor{textcolor}{HTML}{595959}
\definecolor{tuescarlet}{RGB}{200,24,24}

\newtheorem{axiom}{Axiom}[section]

\newtheorem{definition}[axiom]{Definition}

\newtheorem{lemma}[axiom]{Lemma}

\newtheorem{theorem}[axiom]{Theorem}

\title{Fair Schedules for Single Round Robin Tournaments with Ranked Participants}
\author[1]{Sten Wessel}
\author[1]{Cor Hurkens}
\author[1]{Frits Spieksma}
\affil[1]{Department of Mathematics and Computer Science, Eindhoven University of Technology, The Netherlands}
\affil[ ]{\normalsize\texttt{\{s.wessel,c.a.j.hurkens,f.c.r.spieksma\}@tue.nl}}
\date{}

\begin{document}
\maketitle

\begin{abstract}\noindent
    We introduce a new measure to capture fairness of a schedule in a single round robin (SRR) tournament when participants are ranked by strength. To prevent distortion of the outcome of an SRR tournament as well as to guarantee equal treatment, we argue that each participant should face its opponents when ranked by strength in an alternating fashion with respect to the home/away advantage. Here, the home/away advantage captures a variety of situations.
    We provide an explicit construction proving that so-called \emph{ranking-fair} schedules exist when the number of participants is a multiple of 4. Further, we give a formulation that outputs ranking-fair schedules when they exist. Finally, we show that the most popular method to come to a schedule for an SRR tournament, does not allow ranking-fair schedules when the number of teams exceeds 8. These findings impact the type of schedules to be used for SRR tournaments.

    \bigskip

\end{abstract}

\setcounter{section}{-1}
\section{Prologue}
\label{sec:prologue}
The Tata Steel Chess tournament is a chess contest where the fourteen best players in the world (measured by their so-called Elo rating) are invited to play a single round robin tournament.
In the 2002 edition of this tournament (then called the Corus Chess Tournament), the \#2 player of the world, called Michael Adams, faced the strongest 7 other players while playing black, and the remaining 6 weakest players while playing white.
This imbalance had the  potential to distort the outcome of this tournament, and is an extreme example of a more frequently occurring situation where the schedule may favor some players over others.
We show how to remedy this, and arrive at fair schedules for single round robin tournaments where a ranking of the players is prespecified.

\section{Introduction}
\label{sec:intro}
There are two basic formats when designing (sports) competitions: round robin and knock out.
In round robin tournaments, every pair of teams (or players) meets, while in knock out tournaments a team that loses a game is immediately eliminated from the tournament.
In practice, all kinds of variations of these basic formats are used; in particular, a format that first uses a round robin for a number of groups of teams, and then a knockout for the group winners is a popular mix of these two basic formats.
We refer to Devriesere et al.~\cite{devcsagoo2024} and Ribeiro et al.~\cite{riburrwer2023} and the references contained therein for more information on tournament design.

Our focus is exclusively on single round robin tournaments; these are tournaments where every pair of teams meets exactly once.
In addition, we assume that there is a given ranking of the teams ordering them according to their strength.
In many cases such a ranking is explicitly known; indeed, for most sports, there is a ranking of the best players or teams (e.g., the Elo rating in chess, the ATP/WTA ranking in tennis, the FIFA World ranking in football, etc.).

Finally, we assume the existence of an ``asymmetry'' in any particular game; by an asymmetry, we mean the existence of a (dis)advantage in a game for one player compared to the other player. For instance, in the game of chess, playing with white is considered advantageous compared to playing with black.
In fact, more generally, when a match is played in either of the team's venues, there is the so-called ``home advantage''; an effect that is well studied in literature. It refers to the observation that the team playing home has an increased probability of winning the match. We mention the following studies (without aiming to be complete): Pollard and Pollard~\cite{polpol2005} for an overview, Benz et al.~\cite{benblilop2024} deal with American Football, Cabarkapa et al.~\cite{cabdeacic2023} deal with basketball, Pollard~\cite{pol2008} deals with football, Losak and Sabel~\cite{lossab2021} deal with baseball, Adie et al.~\cite{adirenpol2022} deal with cricket, and Alexandros et al.~\cite{alepanmil2012} deal with volleyball. In all these cases, a home advantage is found to be present with the interesting exception that, for the case of cricket, a disadvantage for the home team is found, attributed to the referee's intention to be fair in spite of the home crowd~\cite{adirenpol2022}. In any case, summarizing this discussion, in single round robin tournaments that are played in venues that correspond to the participating teams, the presence of an asymmetry due to home (dis)advantage is well-motivated.

Although we focus on the single round robin format, our results directly generalize to tournaments where teams meet each other an odd number of times.
Also when the tournament is in double round robin format, where every pair of teams meet exactly twice, our results are applicable when the schedule consists of two halves that are scheduled as a single round robin.

To state our contributions, we proceed by considering the schedule of a single round robin tournament.
For reasons of convenience, we refer to the asymmetry by using home/away terminology, i.e., the presence or absence of the home advantage; however, we emphasize that this home/away terminology represents a variety of situations such as playing with white or black in chess, or having the right to bat last, or serve first.

Now, when considering a schedule, focus in literature has always been on balancing home and away games over the rounds, i.e., to balance the asymmetry over time.
Indeed, ideally, a team plays alternatingly home and away.
Achieving this property has been an important driver of many optimization approaches in scheduling round robin tournaments (see Elf et al.~\cite{elfjunrin2003}, Rasmussen and Trick~\cite{rastri2007} and van 't Hof et al.~\cite{vanposbri2010}).

In this paper, we argue that it is also of fundamental importance to balance the home advantage over the opponents {\em ranked by strength}. Indeed, we claim that balancing the home advantage evenly over the ranked opponents is needed to ensure that
\begin{itemize}
    \item[(i)] no team is disadvantaged compared to other teams (achieving {\em fairness}),
    \item[(ii)] no distortion of the final result is present (achieving {\em efficacy}).
\end{itemize}
Ad(i) Indeed, suppose that a strong team faces the strongest half of their opponents playing away, and the weakest half of the opponents playing home (see Section~\ref{sec:prologue} for a real-life example of this phenomenon).
In many cases, this strong team is severely disadvantaged, as the team is likely to win against the weaker teams irrespective of playing home or away, while the team will struggle against the stronger teams now that the team is playing away against all other strong teams.
Or, alternatively, suppose that a weak team faces the weakest half of the opponents at home, and the strongest half of the opponents away.
Again, one can hypothesize that this team has an advantage over other weak teams, as it faces all its ``colleague'' weak teams at home, while playing away for the games that in all likelihood will be lost anyway.
Each of these examples shows that some teams can have an advantage over other teams when there is no balance of the asymmetry over the opponents ranked by strength.

Ad(ii) Efficacy refers to the capability of a tournament to generate a ranking of the teams that is consistent with their strength (see, e.g., Szilai et al.~\cite{szibircsa2022}). Indeed, from an organizer's point of view, it makes sense to construct a schedule where each team faces its opponents in a way where the home/away asymmetry is spread evenly over the opponents ranked by strength.
In the absence of this balance, the final standings may not represent the true strength of the participants, and the tournament's efficacy is low.

Here, we explore the issue of generating schedules where each team's opponents when ranked by strength have the home advantage alternatingly; we provide the following results:
\begin{itemize}
\item We define the notion of a ranking-fair schedule (Section~\ref{sec:probdes}), and introduce a measure capturing to what extent a single round robin schedule is ranking-fair (Section~\ref{sec:measure}).
\item As our main result, we provide an explicit construction yielding a ranking-fair schedule when the number of teams is a multiple of four. Further, we give a mathematical programming formulation that can be used to find ranking-fair schedules when they exist (Section~\ref{sec:results}).
\item We consider the so-called canonical pattern set (CPS), and prove that the CPS does not allow ranking-fair schedules when the number of teams exceeds 8 (Section~\ref{sec:cps}).
\end{itemize}
In Section~\ref{sec:practice}, we give some motivating real-life examples coming from the Tata Steel chess tournament, Danish football league, and the Dutch baseball league.

\section{Problem Description} \label{sec:probdes}
We let $n$ denote the number of teams; we will assume $n$ is even.
Consequently, the number of rounds is $n-1$, and we use $R = \set{1, \dots, n - 1}$ to refer to the set of rounds.
We use~$T = \set{1, \dots, n}$ as the set of teams, and we assume that each team has its own {\em venue}. Furthermore, we assume that the teams~$1, \dots, n$ are sorted based on a predetermined ranking, where team~$1$ is the strongest team and team~$n$ is the weakest team.
A \emph{schedule} specifies for every game between teams~$i$ and~$j$ in which round it is played and in which of the two venues it is played, i.e., whether team~$i$ or team~$j$ plays at home.
A schedule is feasible when every team plays exactly one game in every round, and meets all other teams during all rounds.

In this work, we introduce the concept of \emph{ranking fairness} of a schedule.
\begin{definition}
    The \emph{ranking home-away pattern (ranking HAP)} of a team~$i$, $1 \le i \le n$ is a vector $p = (p_1, p_2, \dots, p_{n-1})$, where~$p_j \in \set{H, A}$ specifies whether team~$i$ plays home or away against its~$j$-th strongest opponent.
\end{definition}

In other words, the ranking HAP of a team is a home-away sequence, where the games are ordered by decreasing strength of its opponents.
For example, in a $6$-team competition, if the ranking HAP for team~$3$ is $(H, H, A, A, A)$, which we will also write as $HHAAA$, then this team plays at home against teams~$1$ and $2$, and away against teams~$4$, $5$, and $6$.
As sketched in the introduction section, our aim is to balance the home advantage over the ranked opponents.
We consider the asymmetry to be perfectly balanced when home and away games alternate in the ranking HAP, i.e., the ranking HAP of a team is either $HAHAH\dots$ or~$AHAHA\dots$, as formulated in the next definition.

\begin{definition}
    A schedule is \emph{ranking-fair} when the ranking HAP of every team is alternating.
\end{definition}
In a ranking-fair schedule, the ranking HAP of team~$1$ either starts with a home game or an away game against team~$2$.
Note that either choice fixes the ranking HAP of team~$2$, as the first venue in its ranking HAP it must be the complement of the first venue of the ranking HAP of team~$1$.
This argument continues to fix the ranking HAPs of the lower-strength teams, and therefore, there are merely two options for the ranking HAPs of all teams of any feasible ranking-fair schedule.
Namely, the one displayed in Table~\ref{tab:rankingHAPset}, where team~$1$ plays team~$2$ at home, or the complement of this schedule, where every home is replaced with an away game and vice versa.

\begin{table}
    \centering
    \caption{Venues for team~$i$ (row) when playing against team~$j$ (column) in a ranking-fair schedule.}\label{tab:rankingHAPset}
    \begin{tabular}{c|ccccccc}
        \toprule
        $i$ \textbackslash\ $j$ & $1$ & $2$ & $3$ & $4$ & $\cdots$ & $n-1$ & $n$ \\
        \midrule
        $1$ & -- & $H$ & $A$ & $H$ & $\cdots$ & $A$ & $H$ \\
        $2$ & $A$ & -- & $H$ & $A$ & $\cdots$ & $H$ & $A$ \\
        $3$ & $H$ & $A$ & -- & $H$ & $\cdots$ & $A$ & $H$ \\
        $4$ & $A$ & $H$ & $A$ & -- & $\cdots$ & $H$ & $A$ \\
        $\vdots$ & $\vdots$ & $\vdots$ & $\vdots$ && $\ddots$ && $\vdots$ \\
        $n-1$ & $H$ & $A$ & $H$ & $A$ & $\cdots$ & -- & $H$ \\
        $n$ & $A$ & $H$ & $A$ & $H$ & $\cdots$ & $A$ & -- \\
        \bottomrule
    \end{tabular}
\end{table}

In the remainder of this work, we assume that a ranking-fair schedule has ranking HAPs according to Table~\ref{tab:rankingHAPset}, i.e., team~$1$ plays team~$2$ at home, as otherwise we can always obtain such a schedule by inverting the venues for every match.
Notice that in a ranking-fair schedule, the venue of a match is completely determined by the \emph{parity} of the two teams, as summarized in the following lemma.
\begin{lemma} \label{lem:rf-parity}
    In a ranking-fair schedule, consider the match between the two teams~$i, j \in T$.
    If~$i$ and $j$ have the same parity (i.e., both odd or both even), the strongest team plays away.
    If~$i$ and~$j$ have a different parity, the strongest team plays at home.
\end{lemma}
It follows directly that a schedule is ranking-fair if and only if the schedule satisfies Lemma~\ref{lem:rf-parity} for all matches.

A fundamental, well-studied property of SRR schedules is the number of so-called \emph{breaks} that occur in a schedule.
Given a schedule, the \emph{home-away pattern} (HAP) of a team is defined as the vector~$h = (h_1, \dots, h_{n-1})$, where~$h_r \in \set{H, A}$ specifies whether that team plays home or away in round~$r \in R$.
We consider the HAPs to be circular, and thus to simplify notation we define~$h_0 \coloneqq h_{n-1}$.

A team has a \emph{break} in round $r \in R$ when $h_{r - 1} = h_r$.
When~$h_{r-1} = h_r = H$, we call the break a \emph{home break}, otherwise it is an \emph{away break}.
The \emph{HAP set} of a schedule denotes the collection of HAPs of all teams.
We call a HAP set \emph{feasible} if there exists a feasible schedule corresponding to this HAP set.

Since the number of rounds is odd and the fact that we consider a HAP to be circular, a HAP needs to have at least one break.
We call a HAP \emph{single-break} if it has exactly one break.
Similarly, we call a HAP set single-break if all HAPs in the HAP set are single-break.

Single-break HAP sets have received considerable attention in the existing literature (De Werra~\cite{dew1981}, Rasmussen and Trick~\cite{rastri2008}, Goossens and Spieksma~\cite{goospi2011}, Lambers et al.~\cite{lamgoospi2023}).
De Werra~\cite{dew1981} showed that any feasible single-break HAP set $\mathcal H$ is \emph{complementary}, i.e., for every pattern in~$\mathcal H$, also the complement pattern is in~$\mathcal H$, with all home games replaced by away games and vice versa.
Also notice that no two teams in feasible schedule can have identical HAPs, as then they would never be able to play each other.
As a consequence, in each round either zero or two HAPs will have a break, and thus the breaks in a feasible single-break schedule occur in exactly~$n/2$ different rounds.
Let us denote these by~$r_1, \dots, r_{n/2}$, sorted in increasing order.
Now, we define~$d_i \coloneqq r_{i + 1} - r_i$ for~$i = 1, \dots, n/2$, where~$r_{n/2+1} \coloneqq r_1 + n - 1$.
The vector~$D = (d_1, \dots, d_{n/2})$ specifies the gaps between the breaks, and we refer to this as the \emph{break-gap representation} or the \emph{D-sequence} of the single-break HAP set.
Note that a D-sequence is feasible if and only if all cyclic and/or reflection permutations of the D-sequence are feasible, as the rounds of a round-robin schedule can be cyclically shifted or played in reverse other (see Lambers et al.~\cite{lamgoospi2023}).
Hence, without loss of generality, we always consider the lexicographically largest sequence among the set of equivalent D-sequences.

One popular HAP set in practice is the so-called canonical pattern set (CPS), which has the D-sequence $22\dots21$ (De Werra~\cite{dew1980}).
The CPS is widely used in practice to schedule competitions.
For example, the international chess federation (FIDE) uses the so-called Berger tables to schedule tournaments, which are based on the CPS\@.
We dedicate Section~\ref{sec:cps} to ranking fairness of schedules corresponding to this particular HAP set.

\section{A Measure for Ranking Fairness} \label{sec:measure}
When a schedule is not ranking-fair, it is relevant to have a measure that represents the amount of imbalance.
We express the ranking fairness of a schedule by measuring how much the ranking HAP deviates from the HAP that is perfectly alternating.
For a team~$t$, we consider all sub-patterns of the ranking HAP, i.e., each consecutive sub-interval of the ranking HAP characterized by the leftmost and rightmost positions~$[i, j]$, where~$1 \le i < j \le n - 1$.
Let~$H^t_{i,j}$ denote the number of home games of team~$t$ that occur in this sub-pattern.
In a ranking-fair schedule, $H^t_{i,j} \in \left\{ \left\lfloor \frac{j - i + 1}{2} \right\rfloor, \left\lceil \frac{j - i + 1}{2} \right\rceil \right\}$, i.e., (approximately) half of the games within each sub-pattern are home games.
To measure the fairness of any schedule, we calculate for every sub-pattern how much the number of home games deviates from this.
Let
\begin{equation}\label{eq:deltat}
    \Delta_t = \sum_{i=1}^{n-2}\sum_{j=i+1}^{n-1} \left| H_{i,j}^t - \frac{j - i + 1}{2} \right|
\end{equation}
denote the sum of the deviations between the actual and ideal number of home games for every sub-pattern in the ranking HAP\@.
We observe that the only minimizers of~$\Delta_t$ are the ranking HAPs that are perfectly alternating between home and away, i.e., a ranking-fair HAP\@.
The measure~$\Delta_t$ is maximal for, among others, the pattern that starts with~$n/2$ home games against the strongest opponents, followed by~$n/2-1$ away games against the weakest opponents.
We normalize~$\Delta_t$ to ensure that the fairness measure always assumes values between zero and one, such that zero yields a ranking-fair HAP and one coincides with a ranking HAP that maximizes~$\Delta_t$.
Our fairness measure~$F_t$ for the ranking HAP of team~$t$ is then defined as
\begin{equation}
\label{eq:Ft}
    F_t = \frac{\Delta_t - \frac 1 8 (n-2)^2}{\frac{1}{24} n (n-1) (n-2)}\,.
\end{equation}
The correction term $(n-2)^2/8$ ensures that a ranking-fair schedule yields~$F_t = 0$, as without correction every odd-length interval adds~$1/2$ to the measure.
The normalization ensures that the maximum of~$F_t$ is~$1$.
We are now able to define the ranking fairness of a schedule as follows:
\begin{definition}
The ranking fairness of a schedule equals $F = \frac 1 n \sum_{t \in T} F_t$.
\end{definition}
In other words, the ranking fairness of a schedule equals the average fairness value of the teams.

While one can think of various ways of measuring the deviation of an arbitrary ranking HAP from the alternating ranking HAP, we point out that the definition phrased above has attractive properties. Indeed, it is symmetric in the sense that viewing the opponents from weak to strong (instead from strong to weak) has no impact on the value of $F$. Another property is summarized in the lemma below.

\begin{lemma}
    $F = 0$ if and only if the schedule is ranking fair.
\end{lemma}
In other words, ranking-fair schedules minimize~$F$, and it is straightforward to show that ranking-fair schedules are the only minimizers of~$F$.

\section{Results}
\label{sec:results}

In this section, we show in which cases a single-break, ranking-fair schedule exists.
We distinguish between whether the number of teams is a multiple of four, or not.

\subsection{The Case $n=4k$} \label{sec:existence}
Let us first restrict to the case when~$n$ is a multiple of four.
We provide a schedule that is, by construction, ranking-fair, and for which we show that it is indeed feasible and single-break.
To this end, we fix the home and away venues for every game as specified in Table~\ref{tab:rankingHAPset}.
We construct a schedule by giving an explicit table~$S(i, j)$, for teams~$i, j \in T$, $i \neq j$ where~$S(i, j) \in R$ specifies the round in which team~$i$ plays team~$j$.
Note that this table needs to be symmetric, i.e., $S(i, j) = S(j, i)$ for all~$i, j = 1, \dots, n$, $i > j$.
Hence, it suffices to only specify the table entries~$S(i, j)$ with $j > i$.

Define~$\pi_i(j) \coloneqq 1 + (n + 1 - i - j \mod n - 1)$.
Then, $\pi_i$ denotes a circular shift of $(n - 1, n - 2, \dots, 1)$.

We start by first filling rows~$1, 3, 5, \dots, n/2 - 1$ of the table, defined by
\begin{equation}
    S(i, j) = \begin{cases}
                  \pi_i(j) & \text{if $i \le \frac n 2$ is odd, $i < j < n$,} \\
                  \pi_i(i) & \text{if $i \le \frac n 2$ is odd, $j = n$.}
    \end{cases}
\end{equation}
The remaining odd rows, except for~$i = n - 1$, are specified by
\begin{equation}
    S(i, j) = \begin{cases}
                  \pi_i(j) & \text{if $\frac n 2 < i \le n - 3$ is odd, $i + 1 < j < n$,} \\
                  \pi_i(i) & \text{if $\frac n 2 < i \le n - 3$ is odd, $j = i + 1$,} \\
                  \pi_i(i + 1) & \text{if $\frac n 2 < i \le n - 3$ is odd, $j = n$.}
    \end{cases}
\end{equation}
We furthermore set~$S(n-1, n) = 3$.
For the even rows, we set
\begin{equation}
    S(i, j) = \begin{cases}
                  S(i - 1, j + 1) & \text{if~$i$ is even, $j$ is odd, $j > i$,} \\
                  S(i - 1, j - 1) & \text{if~$i$ is even, $j$ is even, $j > i$.}
    \end{cases}
\end{equation}
This finalizes the schedule construction.
See also Table~\ref{tab:schedule4k} for an overview of the schedule generated by~$S$.

\begin{landscape}
    \begin{table}
        \centering
        \footnotesize
        \caption{The construction of a minimum-break ranking-fair schedule for $n = 4k$ teams.
        The cell $S(i, j)$ specifies the round in which the game between $i$ and $j$ is played.
        The symbol~$\at{}$ indicates that team~$i$ plays away (and home otherwise) against team~$j$.}\label{tab:schedule4k}
        \begin{tabular}{l|ll|ll|ll|ll|ll|ll|ll|ll|ll}
            \toprule
            $i$ \textbackslash\ $j$ & $1$ & $2$ & $3$ & $4$ & $5$ & $6$ & $\frac n 2 - 1$ & $\frac n 2$ & $\frac n 2 + 1$ & $\frac n 2 + 2$ & $\frac n 2 + 3$ & $\frac n 2 + 4$ & $n - 5$ & $n - 4$ & $n-3$ & $n-2$ & $n-1$ & $n$ \\\midrule
            $1$ & \nothing & $n-1$ & \at{$n-2$} & $n-3$ & &&&&&&&&&& \at{$4$} & $3$ & \at{$2$} & $1$ \\
            $2$ & \at{$n-1$} & \nothing & $n-3$ & \at{$n-2$} &&&&&&&&&&& $3$ & \at{$4$} & $1$ & \at{$2$} \\ \midrule
            $3$ & $n-2$ & \at{$n-3$} & \nothing & $n-5$ & \at{$n-6$} & $n-7$ &&&&&&& \at{$4$} & $3$ & \at{$2$} & $1$ & \at{$n-1$} & $n-4$ \\
            $4$ & \at{$n-3$} & $n-2$ & \at{$n-5$} & \nothing & $n-7$ & \at{$n-6$} &&&&&&& $3$ & \at{$4$} & $1$ & \at{$2$} & $n-4$ & \at{$n-1$} \\ \midrule
            $5$ & $n-4$ & \at{$n-5$} & $n-6$ & \at{$n-7$} & \nothing & $n-9$ &&&&&&& \at{$2$} & $1$ & \at{$n-1$} & $n-2$ & \at{$n-3$} & $n-8$ \\
            $6$  &&&&&&&&&&&&& $1$ & \at{$2$} & $n-2$ & \at{$n-1$} & $n-8$ &\at{$n-3$} \\ \midrule
            \vphantom{$\frac n 2$}  &&&&&&&&&&&&&&&&&& \\
            \vphantom{$\frac n 2$}  &&&&&&&&&&&&&&&&&& \\ \midrule
            $\frac n 2 - 1$ &&&&&&& \nothing & $3$ & \at{$2$} & $1$ &&&&&&& \at{$\frac n 2 + 3$} & $4$ \\
            $\frac n 2$ &&&&&&& \at{$3$} & \nothing & &&&&&&&& $4$ & \at{$\frac n 2 + 3$} \\ \midrule
            $\frac n 2 + 1$ & $\frac n 2$ & \at{$\frac n 2 - 1$} & & & & \at{$3$} & $2$ & \at{$1$} & \nothing & $n-1$ & \at{$n-3$} & $n-4$ &&& \at{$\frac n 2 + 3$} & $\frac n 2 + 2$ & \at{$\frac n 2 + 1$} & $n-2$ \\
            $\frac n 2 + 2$ & \at{$\frac n 2 - 1$} & $\frac n 2$ &&& \at{$3$} & $4$ & \at{$1$} & $2$ & \at{$n-1$} & \nothing & $n-4$ & \at{$n-3$} & & & $\frac n 2 + 2$ & \at{$\frac n 2 + 3$} & $n-2$ & \at{$\frac n 2 + 1$} \\ \midrule
            \vphantom{$\frac n 2$} &&&&&&&&&&&&&&&&&& \\
            \vphantom{$\frac n 2$} &&&&&&&&&&&&&&&&&& \\ \midrule
            $n - 5$ &&&&&&&&&&&&& \nothing & $11$ & \at{$9$} & $8$ & \at{$7$} & $10$ \\
            $n - 4$ &&&&&&&&&&&&& \at{$11$} & \nothing & $8$ & \at{$9$} & $10$ & \at{$7$} \\ \midrule
            $n - 3$ & $4$ & \at{$3$} & $2$ & \at{$1$} & $n-1$ & \at{$n-2$} &&&&&&& $9$ & \at{$8$} & \nothing & $7$ & \at{$5$} & $6$ \\
            $n - 2$ & \at{$3$} & $4$ & \at{$1$} & $2$ & \at{$n-2$} & $n-1$ &&&&&&& \at{$8$} & $9$ & \at{$7$} & \nothing & $6$ & \at{$5$}  \\ \midrule
            $n - 1$ & $2$ & \at{$1$} & $n-1$ & \at{$n-4$} & $n-3$ & \at{$n-8$} &&&&&&& $7$ & \at{$10$} & $5$ & \at{$6$} & \nothing & $3$ \\
            $n$ & \at{$1$} & $2$ & \at{$n-4$} & $n-1$ & \at{$n-8$} & $n-3$ & $4$ &&&&&& \at{$10$} & $7$ & \at{$6$} & $5$ & \at{$3$} & \nothing \\ \bottomrule
        \end{tabular}
    \end{table}
\end{landscape}

We first show that the schedule specified by~$S$ can in fact be played.
\begin{lemma}\label{lem:n4k-feasible}
    The schedule specified by~$S$ is feasible.
\end{lemma}
\begin{proof}
    Each team should play exactly one game in every round, each against a different opponent.
    Thus, every row of the table must contain all rounds~$1, \dots, n-1$ exactly once.
    We consider the following cases:
    \begin{itemize}
        \item Team~$i$ is odd, $i \le n/2$.
        We split the row in two parts.
        The last part, for indices~$j$ with~$i < j \le n$, contains rounds~$n - 1, n - 2, \dots, n - i + 2$ and $1, 2, \dots, n - 2i + 2$, by the definition of~$\pi_i$.
        For the first part, for indices~$j$ with~$1 \le j < i$, we first observe by the symmetry of the table that~$S(i, j) = S(j, i)$, and shift our attention to the~$i$-th column.
        There are~$\frac{i - 1}{2}$ odd indices~$j < i$, and because the rows above~$i$ are a cyclic shift, these cells contain rounds~$n - 2i + 3, n - 2i + 5, \dots, n - i$.
        The even indices~$j$ complete the missing rounds~$n - 2i + 4, n - 2i + 6, \dots, n - i + 1$.
        \item Team~$i$ is even.
        Even rows are a permutation of the odd row directly above, where the rounds of~$(j, j+1)$ ($j$ odd) are permuted.
        Hence, even rows contain all rounds exactly once.
    \end{itemize}
    Hence, each row of~$S$ is a permutation of $(n-1, n-2, \dots, 1)$, completing the proof.
\end{proof}

\begin{lemma}\label{lem:n4k-singlebreak}
    The schedule specified by~$S$ is single-break.
\end{lemma}
\begin{proof}
    By construction, the odd teams play against consecutive opponents in consecutive rounds, except when playing against team~$n$ or team~$i+1$.
    Since the schedule is ranking-fair, by construction, the rounds where the opponent is not $n$ or $i+1$ are played with alternating home and away venues.
    For the even teams~$i$, the schedule is derived from the (odd) team~$i - 1$, where we observe that:
    \begin{itemize}
        \item Team~$i$ for odd~$i$, $i \le n/2$ has a break in round~$S(i, n)$.
        \item Team~$i$ for odd~$i$, $i > n/2$ has a break in round~$S(i, i+1)$.
        \item Team~$i$ for even~$i$, $i \le n/2$ has a break in round~$S(i, i-1)$.
        \item Team~$i$ for even~$i$, $i > n/2$ has a break in round~$S(i, i-3)$. \qedhere
    \end{itemize}
\end{proof}
We thus have our main result.
\begin{theorem}
    For $n=4k$ there exists a minimum break, ranking fair schedule.
\end{theorem}
\begin{proof}
    Consider the schedule generated by~$S$.
    By specifying the home and away venues for every game as in Table~\ref{tab:rankingHAPset}, the schedule is ranking-fair by construction.
    By Lemma~\ref{lem:n4k-feasible} the schedule is feasible and by Lemma~\ref{lem:n4k-singlebreak}, the schedule is single-break, completing the proof.
\end{proof}

From the proof of Lemma~\ref{lem:n4k-singlebreak} one can observe that the breaks occur in rounds~$1, 3, 4, 7, 8, \dots, n-5, n-4, n-1$.
Hence, the D-sequence of this schedule is equal to~$3131\dots312$.
The construction is in the spirit of the popular circle method (Lambrechts et al.~\cite{lambrechts}); see also Figure~\ref{fig:4k:circle}.
We consider teams to be paired as~$(1, 2), (3, 4), \dots, (n-1, n)$.
The teams are seated at seats~$s_1, \dots, s_n$, viewed in a circular fashion, i.e., seat $s_1$ is next to seat~$s_n$.
The teams sit in order of their ranking on the seats, with team~$1$ occupying seat~$s_{2k+1}$ in rounds~$4k+1$ and $4k+2$.
In rounds~$4k+3$ and $4k+4$, team~$1$ occupies seat~$s_{2k+2}$.
As in the regular circle method, teams $i$ and $n+1-i$ meet in the first round.
Note that as $n$ is a multiple of 4, we can consider an even number $n/2$ pairs of players $2i-1,2i$, for $i=1,\dots,n/2$.
Also note that the home assignment is dictated by the ranking HAP;
for equal parity players, the higher indexed player plays at home.
An arc from $s$ to $t$ denotes that the player on seat $s$ plays at home against the player on seat $t$.
The breaks are experienced in rounds $4k+3$ at seats $s_n$ and $s_{\frac{n}2+1}$, and in rounds $4k+4$ at seats $s_1$ and $s_{\frac{n}2}$.

\begin{figure}[p]
    \centering
    \begin{subfigure}{\textwidth}
        \centering
        \medskip
\begin{tikzpicture}[xscale=.75, yscale=1]
    \draw[blue] (0,1) -- (16,1) .. controls (16.5,0.5) .. (16,0) -- (0,0) .. controls (-0.5,0.5) .. cycle;
    \draw (0 cm,1pt) node[anchor=north] {$s_{n}$};

    \foreach \x in {1,2,3,4,5}
        \draw (\x cm-1cm,1cm-1pt) node[anchor=south] {$s_{\x}$};

    \foreach \x in {1,2,3,4}
        \draw (\x cm,1pt) node[anchor=north] {$s_{n-\x}$};

    \draw (16 cm,1cm-1pt) node[anchor=south] {$s_{\frac{n}2}$};

    \foreach \x in {1,2,3,4}
        \draw (16cm - \x cm,1cm-1pt) node[anchor=south] {$s_{\frac{n}2-\x}$};

    \foreach \x in {1,2,3,4,5}
        \draw (17cm - \x cm,1pt) node[anchor=north] {$s_{\frac{n}2+\x}$};

    \draw (6 cm,1cm-1pt) node[anchor=south] {$s_{2k-1}$};
    \draw (7 cm,1cm-1pt) node[anchor=south] {$s_{2k}$};
    \draw (8 cm,1cm-1pt) node[anchor=south] {$s_{2k+1}$};
    \draw (9 cm,1cm-1pt) node[anchor=south] {$s_{2k+2}$};
    \draw (10 cm,1cm-1pt) node[anchor=south] {$s_{2k+3}$};
    \draw (8 cm,1pt) node[anchor=north] {$s_{n-2k}$};
    \draw (9.25 cm,1pt) node[anchor=north] {$s_{n-2k-1}$};
    \draw (6 cm,1.5cm-1pt) node[anchor=south] {$p_{n-1}$};
    \draw (7 cm,1.5cm-1pt) node[anchor=south] {$p_n$};
    \draw (8 cm,1.5cm-1pt) node[anchor=south] {$p_1$};
    \draw (9 cm,1.5cm-1pt) node[anchor=south] {$p_2$};
    \draw (10 cm,1.5cm-1pt) node[anchor=south] {$p_3$};
    \draw (8 cm,-0.5cm+1pt) node[anchor=north] {$p_{n-4k}$};
    \draw (9.25 cm,-0.5cm+1pt) node[anchor=north] {$p_{n-4k-1}$};

    \foreach \x in {1,2,3,4,5,7,8,9,10,11,13,14,15,16,17}
        \filldraw[gray] (\x cm-1cm,1cm) circle [radius=2pt];

    \foreach \x in {1,2,3,4,5,7,8,9,10,11,13,14,15,16,17}
        \filldraw[gray] (\x cm-1cm,0cm) circle [radius=2pt];

    \foreach \x in {1.5,3.5,5.5,7.5, 9.5,12.5,14.5}
        \draw[dashed] (\x cm,-0.5cm) -- (\x cm,1.5cm);

    \draw (16 cm,1 cm) node[anchor=west] {$p_{\frac{n}2-2k}$};
    \draw (16 cm,0 cm) node[anchor=west] {$p_{\frac{n}2-2k+1}$};

    \foreach \x in {8,9,10,12,13,14,15,16}
        \draw[very thick,->] (\x cm,1cm-1pt) -- (\x cm,0cm+1pt);

    \foreach \x in {0,2,6}
        \draw[very thick,->] (\x cm+1cm-1pt,1cm-1pt) -- (\x cm+1pt,0cm+1pt);

    \foreach \x in {0,2,6}
        \draw[very thick,->] (\x cm+1pt,1cm-1pt) --(\x cm+1cm-1pt,0cm+1pt);
\end{tikzpicture}
        \caption{Round $4k+1$}
    \end{subfigure}
    \begin{subfigure}{\textwidth}
        \centering
        \medskip
\begin{tikzpicture}[xscale=.75, yscale=1]
    \draw[blue] (0,1) -- (16,1) .. controls (16.5,0.5) .. (16,0) -- (0,0) .. controls (-0.5,0.5) .. cycle;
    \draw (0 cm,1pt) node[anchor=north] {$s_{n}$};

    \foreach \x in {1,2,3,4,5}
       \draw (\x cm-1cm,1cm-1pt) node[anchor=south] {$s_{\x}$} ;
    \foreach \x in {1,2,3,4}
       \draw (\x cm,1pt) node[anchor=north] {$s_{n-\x}$} ;

    \draw (16 cm,1cm-1pt) node[anchor=south] {$s_{\frac{n}2}$};

    \foreach \x in {1,2,3,4}
       \draw (16cm - \x cm,1cm-1pt) node[anchor=south] {$s_{\frac{n}2-\x}$} ;
    \foreach \x in {1,2,3,4,5}
       \draw (17cm - \x cm,1pt) node[anchor=north] {$s_{\frac{n}2+\x}$} ;

    \draw (6 cm,1cm-1pt) node[anchor=south] {$s_{2k-1}$};
    \draw (7 cm,1cm-1pt) node[anchor=south] {$s_{2k}$};
    \draw (8 cm,1cm-1pt) node[anchor=south] {$s_{2k+1}$};
    \draw (9 cm,1cm-1pt) node[anchor=south] {$s_{2k+2}$};
    \draw (10 cm,1cm-1pt) node[anchor=south] {$s_{2k+3}$};
    \draw (8 cm,1pt) node[anchor=north] {$s_{n-2k}$};
    \draw (9.25 cm,1pt) node[anchor=north] {$s_{n-2k-1}$};
    \draw (6 cm,1.5cm-1pt) node[anchor=south] {$p_{n-1}$};
    \draw (7 cm,1.5cm-1pt) node[anchor=south] {$p_n$};
    \draw (8 cm,1.5cm-1pt) node[anchor=south] {$p_1$};
    \draw (9 cm,1.5cm-1pt) node[anchor=south] {$p_2$};
    \draw (10 cm,1.5cm-1pt) node[anchor=south] {$p_3$};
    \draw (8 cm,-0.5cm+1pt) node[anchor=north] {$p_{n-4k}$};
    \draw (9.25 cm,-0.5cm+1pt) node[anchor=north] {$p_{n-4k-1}$};

    \foreach \x in {1,2,3,4,5,7,8,9,10,11,13,14,15,16,17}
       \filldraw[gray] (\x cm-1cm,1cm) circle [radius=2pt];
    \foreach \x in {1,2,3,4,5,7,8,9,10,11,13,14,15,16,17}
       \filldraw[gray] (\x cm-1cm,0cm) circle [radius=2pt];
    \foreach \x in {1.5,3.5,5.5,7.5, 9.5,12.5,14.5}
       \draw[dashed] (\x cm,-0.5cm) -- (\x cm,1.5cm) ;

    \draw (16 cm,1 cm) node[anchor=west] {$p_{\frac{n}2-2k}$};
    \draw (16 cm,0 cm) node[anchor=west] {$p_{\frac{n}2-2k+1}$};

    \foreach \x in {2,3,4,6,7}
        \draw[very thick,->] (\x cm,0cm+1pt) -- (\x cm-2cm,1cm-1pt) ;
    \foreach \x in {0,1}
        \draw[red,very thick,->] (\x cm,0cm+1pt) -- (\x cm+6cm,1cm-1pt) ;
    \foreach \x in {8,10,13,15}
        \draw[very thick,->] (\x cm+1pt,0cm+1pt) -- (\x cm+1cm-1pt,1cm-1pt) ;
    \foreach \x in {8,10,13,15}
        \draw[very thick,->] (\x cm+1cm-1pt,0cm+1pt) -- (\x cm+1pt,1cm-1pt) ;
\end{tikzpicture}
        \caption{Round $4k+2$}
    \end{subfigure}
    \begin{subfigure}{\textwidth}
        \centering
        \medskip
\begin{tikzpicture}[xscale=.75, yscale=1]
    \draw[blue] (0,1) -- (16,1) .. controls (16.5,0.5) .. (16,0) -- (0,0) .. controls (-0.5,0.5) .. cycle;
    \draw (0 cm,1pt) node[anchor=north] {$s_{n}$};

    \foreach \x in {1,2,3,4,5}
       \draw (\x cm-1cm,1cm-1pt) node[anchor=south] {$s_{\x}$} ;
    \foreach \x in {1,2,3,4}
       \draw (\x cm,1pt) node[anchor=north] {$s_{n-\x}$} ;

    \draw (16 cm,1cm-1pt) node[anchor=south] {$s_{\frac{n}2}$};

    \foreach \x in {1,2,3,4}
       \draw (16cm - \x cm,1cm-1pt) node[anchor=south] {$s_{\frac{n}2-\x}$} ;
    \foreach \x in {1,2,3,4,5}
       \draw (17cm - \x cm,1pt) node[anchor=north] {$s_{\frac{n}2+\x}$} ;

    \draw (6 cm,1cm-1pt) node[anchor=south] {$s_{2k-1}$};
    \draw (7 cm,1cm-1pt) node[anchor=south] {$s_{2k}$};
    \draw (8 cm,1cm-1pt) node[anchor=south] {$s_{2k+1}$};
    \draw (9 cm,1cm-1pt) node[anchor=south] {$s_{2k+2}$};
    \draw (10 cm,1cm-1pt) node[anchor=south] {$s_{2k+3}$};
    \draw (8 cm,1pt) node[anchor=north] {$s_{n-2k}$};
    \draw (9.25 cm,1pt) node[anchor=north] {$s_{n-2k-1}$};
    \draw (6 cm,1.5cm-1pt) node[anchor=south] {$p_{n-2}$};
    \draw (7 cm,1.5cm-1pt) node[anchor=south] {$p_{n-1}$};
    \draw (8 cm,1.5cm-1pt) node[anchor=south] {$p_n$};
    \draw (9 cm,1.5cm-1pt) node[anchor=south] {$p_1$};
    \draw (10 cm,1.5cm-1pt) node[anchor=south] {$p_2$};
    \draw (8 cm,-0.5cm+1pt) node[anchor=north] {$p_{n-4k-1}\quad$};
    \draw (9.25 cm,-0.5cm+1pt) node[anchor=north] {$p_{n-4k-2}$};

    \foreach \x in {1,2,3,4,5,7,8,9,10,11,13,14,15,16,17}
       \filldraw[gray] (\x cm-1cm,1cm) circle [radius=2pt];
    \foreach \x in {1,2,3,4,5,7,8,9,10,11,13,14,15,16,17}
       \filldraw[gray] (\x cm-1cm,0cm) circle [radius=2pt];
    \foreach \x in {0.5,2.5,4.5,6.5,8.5, 10.5,11.5,13.5,15.5}
       \draw[dashed] (\x cm,-0.5cm) -- (\x cm,1.5cm) ;

    \draw (0 cm,1 cm) node[anchor=east] {$p_{n-2k}$};
    \draw (0 cm,0 cm) node[anchor=east] {$p_{n-2k-1}$};
    \draw (16 cm,1 cm) node[anchor=west] {$p_{\frac{n}2-2k-1}$};
    \draw (16 cm,0 cm) node[anchor=west] {$p_{\frac{n}2-2k}$};
    \draw[very thick,->] (0 cm,0cm+1pt) -- (0 cm,1cm-1pt) ;

    \foreach \x in {9,10,12,13,14,15,16}
        \draw[very thick,->] (\x cm,1cm-1pt) -- (\x cm,0cm+1pt) ;
    \foreach \x in {1,3,7}
        \draw[very thick,->] (\x cm+1cm-1pt,1cm-1pt) -- (\x cm+1pt,0cm+1pt) ;
    \foreach \x in {1,3,7}
        \draw[very thick,->] (\x cm+1pt,1cm-1pt) -- (\x cm+1cm-1pt,0cm+1pt) ;
\end{tikzpicture}
        \caption{Round $4k+3$}
    \end{subfigure}
    \begin{subfigure}{\textwidth}
        \centering
        \begin{tikzpicture}[xscale=.75, yscale=1]
    \draw[blue] (0,1) -- (16,1) .. controls (16.5,0.5) .. (16,0) -- (0,0) .. controls (-0.5,0.5) .. cycle;
    \draw (0 cm,1pt) node[anchor=north] {$s_{n}$};
    \foreach \x in {1,2,3,4,5}
       \draw (\x cm-1cm,1cm-1pt) node[anchor=south] {$s_{\x}$} ;
    \foreach \x in {1,2,3,4}
       \draw (\x cm,1pt) node[anchor=north] {$s_{n-\x}$} ;
    \draw (16 cm,1cm-1pt) node[anchor=south] {$s_{\frac{n}2}$};
    \foreach \x in {1,2,3,4}
       \draw (16cm - \x cm,1cm-1pt) node[anchor=south] {$s_{\frac{n}2-\x}$} ;
    \foreach \x in {1,2,3,4,5}
       \draw (17cm - \x cm,1pt) node[anchor=north] {$s_{\frac{n}2+\x}$} ;
    \draw (6 cm,1cm-1pt) node[anchor=south] {$s_{2k-1}$};
    \draw (7 cm,1cm-1pt) node[anchor=south] {$s_{2k}$};
    \draw (8 cm,1cm-1pt) node[anchor=south] {$s_{2k+1}$};
    \draw (9 cm,1cm-1pt) node[anchor=south] {$s_{2k+2}$};
    \draw (10 cm,1cm-1pt) node[anchor=south] {$s_{2k+3}$};
    \draw (8 cm,1pt) node[anchor=north] {$s_{n-2k}$};
    \draw (9.25 cm,1pt) node[anchor=north] {$s_{n-2k-1}$};
    \draw (6 cm,1.5cm-1pt) node[anchor=south] {$p_{n-2}$};
    \draw (7 cm,1.5cm-1pt) node[anchor=south] {$p_{n-1}$};
    \draw (8 cm,1.5cm-1pt) node[anchor=south] {$p_n$};
    \draw (9 cm,1.5cm-1pt) node[anchor=south] {$p_1$};
    \draw (10 cm,1.5cm-1pt) node[anchor=south] {$p_2$};
    \draw (8 cm,-0.5cm+1pt) node[anchor=north] {$p_{n-4k-1}\quad$};
    \draw (9.25 cm,-0.5cm+1pt) node[anchor=north] {$p_{n-4k-2}$};
    \foreach \x in {1,2,3,4,5,7,8,9,10,11,13,14,15,16,17}
       \filldraw[gray] (\x cm-1cm,1cm) circle [radius=2pt];
    \foreach \x in {1,2,3,4,5,7,8,9,10,11,13,14,15,16,17}
       \filldraw[gray] (\x cm-1cm,0cm) circle [radius=2pt];
    \foreach \x in {0.5,2.5,4.5,6.5,8.5, 10.5,11.5,13.5,15.5}
       \draw[dashed] (\x cm,-0.5cm) -- (\x cm,1.5cm) ;
    \draw (0 cm,1 cm) node[anchor=east] {$p_{n-2k}$};
    \draw (0 cm,0 cm) node[anchor=east] {$p_{n-2k-1}$};
    \draw (16 cm,1 cm) node[anchor=west] {$p_{\frac{n}2-2k-1}$};
    \draw (16 cm,0 cm) node[anchor=west] {$p_{\frac{n}2-2k}$};
    \foreach \x in {2,3,4,7,8}
        \draw[very thick,->] (\x cm-1pt,0cm+1pt) -- (\x cm-2cm+1pt,1cm-1pt) ;
    \foreach \x in {9,12,14}
        \draw[very thick,->] (\x cm+1pt,0cm+1pt) -- (\x cm+1cm-1pt,1cm-1pt) ;
    \foreach \x in {9,12,14}
        \draw[very thick,->] (\x cm+1cm-1pt,0cm+1pt) -- (\x cm+1pt,1cm-1pt) ;
    \draw[very thick,->,red] (1cm-1pt,0) -- (0+1pt,0);
    \draw[very thick,<-,red] (7cm+1pt,1cm-1pt) .. controls (10,-2) and (16,-2) .. (16,0cm-1pt);
    \draw[very thick,<-,red] (8,1cm+1pt) .. controls (8,3) and (16,3) .. (16,1cm+1pt);
\end{tikzpicture}
        \caption{Round $4k+4$}
    \end{subfigure}
    \caption{The ranking-fair schedule generated by a modified circle method.}\label{fig:4k:circle}
\end{figure}
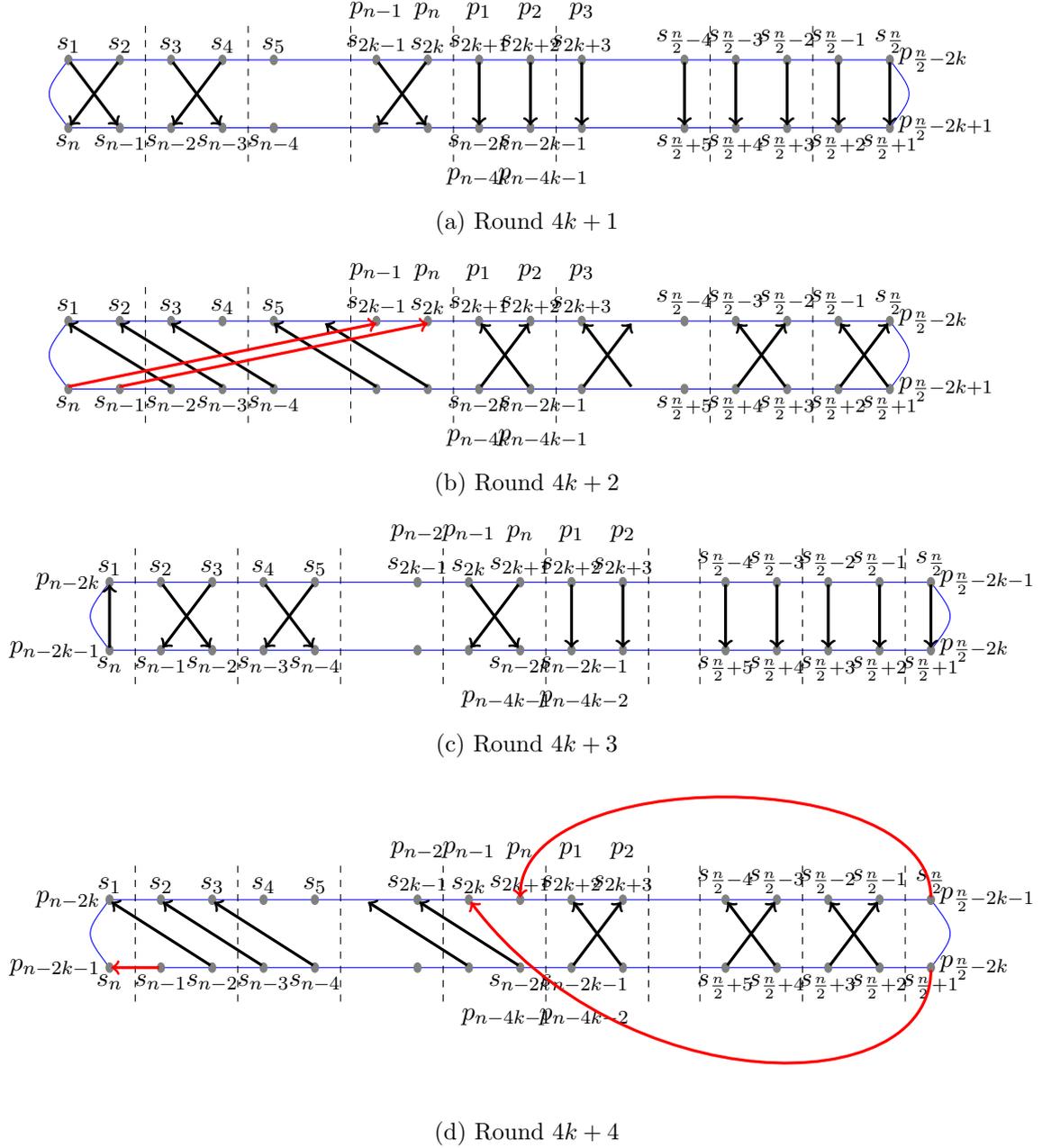

\subsection{The Case $n=4k+2$}
\label{sec:specialcasen}
The construction to generate a ranking-fair schedule in the previous section does not extend to the case where $n = 4k+2$.
In this section, we prove the non-existence of a single-break, ranking-fair schedule when~$n=6$.
In addition, we show computationally that also for $n \in \set{10, 14}$ no single-break, ranking-fair schedules exist.
Finally, we prove that such schedules exist for all other values of $n \le 98$.

\begin{lemma}
    There is no single-break, ranking-fair schedule for~$n=6$.
\end{lemma}
\begin{proof}
    For~$n=6$, the canonical pattern set with D-sequence~$221$ is the only HAP set that is single-break feasible (Lambers et al.~\cite{lamgoospi2023}).
    We will try to construct a schedule using this pattern set that is ranking-fair, where we will eventually reach a contradiction.

    Suppose that we fix the breaks such that the breaks occur in rounds~$1$, $2$, and~$4$.
    For brevity, we call the odd teams~$a, b, c = t(H^1), t(H^2), t(H^4)$ and the even teams~$x, y, z = t(A^1), t(A^2), t(A^4)$.
    We write $\match s t$ to denote the match between teams~$s$ and~$t$, where team~$t$ plays away.

    Now, the match between~$a$ and~$y$ is fixed at round~$1$, as in all other rounds the HAPs are equal.
    Analogously, the match between~$tb$ and~$x$ is fixed at round~$1$, which must fix the match between the remaining teams~$c$ and~$z$ also in round~$1$.
    By nature of the patterns, $a$, $b$, and $c$ play at home in round~$1$.
    Since the matches are between teams of different parity, the team playing at home is stronger than the team playing away (Lemma~\ref{lem:rf-parity}).
    Hence, $a < y$, $b < x$, and $c < z$.

    Now, we consider the following cases for which pattern gets assigned to the strongest team.
    See also Figure~\ref{fig:proofn6-not-rf}, visualizing the fixed matches in the cases below.
    \begin{itemize}
        \item Suppose~$a = 1$.
        Since team~$1$ plays away against all odd opponents, this fixes the match~$\match c a $ in round~$4$, which in turn fixes~$\match b a$ in round~$2$.
        This fixes~$\match c b$ in round~$3$, yielding~$b < c$.
        Now the assignment of the teams to patterns is fully determined as~$a < y < b < x < c < z$.
        Team~$y$ plays at home in round~$3$, against either~$x$ or~$z$.
        But since team~$y$ is stronger than both possible opponents, it must play away against both of them, making the schedule infeasible.

        \item Suppose~$b = 1$.
        The argument is analogous to the previous case, with the following fixings of matches: $\match c b$ in round~$3$, $\match a b$ in round~$5$, $\match c a$ in round~$4$.
        Then, $b < x < a < y < c < z$, but in round~$4$ team~$x$ plays at home against either~$y$ or~$z$.
        Since team~$x$ is stronger than both even opponents, it must play away against both teams.

        \item Suppose~$c = 1$.
        This fixes matches $\match b c$ in round~$2$, $\match a c$ in round~$5$, $\match c x$ in round~$3$, and $\match c y$ in round~$4$.
        Now, $x$ plays against $a$ in either round $2$ or $4$, with~$x$ playing at home in both.
        Thus, $x < a$.
        This fully determines the assignment of teams to the patterns, as~$c < z < b < x < a < y$.
        But also $y$ plays~$b$ at home, in either round $3$ or $5$, yielding~$y < b$, a contradiction.\qedhere
    \end{itemize}
\end{proof}

\begin{figure}
    \centering
    \setlength\tabcolsep{1ex}
    \begin{subfigure}{.3\textwidth}
        \centering
        \footnotesize
        \begin{tabular}{c|ccccc}
            \toprule
                & 1 & 2 & 3 & 4 & 5 \\
            \midrule
            $a$ & \ccell{yellow!50}{H} & \ccell{basegreen!50}A & H & \ccell{basegreen!50}A & H \\
            $b$ & \ccell{baseorange!50}{H} & \ccell{basegreen!50}H & \ccell{baseblue!50}A & H & A \\
            $c$ & \ccell{tuescarlet!50}{H} & A & \ccell{baseblue!50}H & \ccell{basegreen!50}H & A \\
            $x$ & \ccell{baseorange!50}{A} & H & A & H & A \\
            $y$ & \ccell{yellow!50}{A} & A & H & A & H \\
            $z$ & \ccell{tuescarlet!50}{A} & H & A & A & H \\
            \bottomrule
        \end{tabular}
        \caption{The case~$a=1$.}
    \end{subfigure}
    \begin{subfigure}{.3\textwidth}
        \centering
        \footnotesize
        \begin{tabular}{c|ccccc}
            \toprule
            & 1 & 2 & 3 & 4 & 5 \\
            \midrule
            $a$ & \ccell{yellow!50}H & A & H & \ccell{baseblue!50}A & \ccell{basegreen!50}H \\
            $b$ & \ccell{baseorange!50}H & H & \ccell{basegreen!50}A & H & \ccell{basegreen!50}A \\
            $c$ & \ccell{tuescarlet!50}H & A & \ccell{basegreen!50}H & \ccell{baseblue!50}H & A \\
            $x$ & \ccell{baseorange!50}A & H & A & H & A \\
            $y$ & \ccell{yellow!50}A & A & H & A & H \\
            $z$ & \ccell{tuescarlet!50}A & H & A & A & H \\
            \bottomrule
        \end{tabular}
        \caption{The case~$b=1$.}
    \end{subfigure}
    \begin{subfigure}{.3\textwidth}
        \centering
        \footnotesize
        \begin{tabular}{c|ccccc}
            \toprule
            & 1 & 2 & 3 & 4 & 5 \\
            \midrule
            $a$ & \ccell{yellow!50}H & A & H & A & \ccell{basegreen!50}H \\
            $b$ & \ccell{baseorange!50}H & \ccell{basegreen!50}H & A & H & A \\
            $c$ & \ccell{tuescarlet!50}H & \ccell{basegreen!50}A & \ccell{basegreen!50}H & \ccell{basegreen!50}H & \ccell{basegreen!50}A \\
            $x$ & \ccell{baseorange!50}A & H & \ccell{basegreen!50}A & H & A \\
            $y$ & \ccell{yellow!50}A & A & H & \ccell{basegreen!50}A & H \\
            $z$ & \ccell{tuescarlet!50}A & H & A & A & H \\
            \bottomrule
        \end{tabular}
        \caption{The case~$c=1$.}
    \end{subfigure}
    \caption{The fixed matches in a single-break, ranking-fair schedule for~$n=6$, depending on which pattern is assigned to team~$1$.}\label{fig:proofn6-not-rf}
\end{figure}

To settle the existence of ranking-fair schedules when the number of teams is at least $10$, we define an integer programming (IP) model that, given a HAP set~$\mathcal H$, generates a ranking-fair schedule corresponding to this HAP set, if it exists.
Denote by~$\mathcal H_H$ and $\mathcal H_A$ the HAPs that have a home-break and an away-break, respectively.
Note that odd teams need to be assigned a pattern from~$\mathcal H_H$, while even teams are assigned a HAP from~$\mathcal H_A$.

For a HAP~$p \in \mathcal H$, let~$t(p)$ denote the team that is assigned to play with pattern~$p$.
We introduce variables~$x_{\set{p,q},r} \in \set{0, 1}$ for each unordered pair of patterns~$\set{p, q} \in \binom{\mathcal H}{2}$ and every round~$r \in R$, which denote that $t(p)$ plays against $t(q)$ in round~$r$.
Note that~$p_r$ and $q_r$ dictate whether $t(p)$ or $t(q)$ plays at home in round~$r$.
We define for, any $p, q \in \mathcal H$, the set~$R_{p,q} \subseteq R$ as the set of rounds where $t(p)$ can play against $t(q)$ where $t(p) < t(q)$.
Note that, in a ranking-fair schedule, when~$t(p)$ is stronger than~$t(q)$ and both~$p$ and~$q$ have a home-break or both have an away-break, $t(p)$ has to play at home against $t(q)$.
Similarly, when only one of the patterns $p$ and $q$ has a home-break (and the other an away-break), $t(p)$ plays away against $t(q)$.
Thus, we can define
\begin{equation}
    R_{p,q} = \begin{cases}
                  \set{r \in R : p_r = H, q_r = A} & \text{if $p,q \in \mathcal H_H$ or $p,q \in \mathcal H_A$,} \\
                  \set{r \in R : p_r = A, q_r = H} & \text{otherwise.}
    \end{cases}
\end{equation}
We also introduce variables~$b_{p, q} \in \set{0, 1}$ for each~$p, q \in \mathcal H$, $p \neq q$, that indicate whether~$t(p)$ is stronger than~$t(q)$, i.e., $b_{p,q} = 1$ if and only if~$t(p) < t(q)$.
Then, the model for finding a ranking-fair schedule corresponding to the given HAP set~$\mathcal H$ reads as follows.
\begin{subequations}
    \label{eq:ip}
    \begin{align}
        &&\text{minimize}\quad && 0 &&& \\
        &&\text{subject to}\quad && \sum_{q \in \mathcal H \setminus\set{p}} x_{\set{p, q}, r} &= 1 &&\qquad \forall p \in \mathcal H, r \in R, && \label{cons:ip:every-round-once} \\
        &&&& \sum_{r \in R_{p,q}} x_{\set{p,q}, r} &= b_{p, q} &&\qquad \forall p,q \in \mathcal H, p \neq q, && \label{cons:ip:link-b-x}\\
        &&&& b_{p,q} + b_{q,r} - b_{p, r} &\le 1 &&\qquad \forall p,q,r \in \mathcal H, p \neq q, p \neq r, q \neq r, && \label{cons:ip:transitivity-b}\\
        &&&& b_{p,q} + b_{q,p} &= 1 &&\qquad \forall p,q \in \mathcal H, p \neq q, && \label{cons:ip:choose-b}\\
        &&&& x_{\set{p, q},r} &\in \set{0, 1} &&\qquad\forall \set{p, q} \in \binom{\mathcal H}{2}, r \in R,\\
        &&&& b_{p,q} &\in \set{0, 1} &&\qquad \forall p, q \in \mathcal H, p \neq q. &&
    \end{align}
\end{subequations}
Constraints~\eqref{cons:ip:every-round-once} ensure that every team plays exactly once in every round.
Constraints~\eqref{cons:ip:link-b-x} ensure that the match between the teams assigned to patterns~$p$ and~$q$ occurs at a round in which the patterns have a complementary venue, while also enforcing that the schedule is ranking-fair, through the definition of~$R_{p,q}$.
Lastly, constraints~\eqref{cons:ip:transitivity-b} and \eqref{cons:ip:choose-b} impose a linear order on the teams assigned to the patterns, from which one can reconstruct the rank of the team assigned to each pattern.
In a feasible solution of the integer program, the rank of the team assigned to pattern~$p$ can be determined as~$t(p) = 1 + \sum_{q \in \mathcal H \setminus\set{p}} b_{q, p}$.
The objective function is constant, as finding a ranking-fair schedule corresponding to HAP set $\mathcal H$ is a feasibility problem.

For $n=6$, $n=10$, and $n=14$, the only feasible D-sequences, obtained from \cite{lamgoospi2023}, yield no ranking-fair schedule.
Hence, no ranking-fair schedules exist in these cases.
However, for $n = 18, 22, 26, \dots, 98$, the IP model finds single-break, ranking-fair schedules.
The particular D-sequence for this schedule is given by $2212 (31)^i 2 (13)^j$, where $i = \left\lceil \frac 1 4 (\frac n 2 - 5) \right\rceil$ and $j = \left\lfloor \frac 1 4 (\frac n 2 - 5) \right\rfloor$, denoting repetitions of the sub-sequences $31$ and $13$, respectively.

\section{The Canonical Pattern Set}
\label{sec:cps}
In practice, the schedule of SRR tournaments is based on a particular HAP set, namely the one with D-sequence $22\dots21$, which is called the canonical pattern set (CPS).
Because schedules corresponding to this HAP set are easy to construct, the CPS is widely used in many competition schedules.
In this section, we ask whether schedules using the CPS can be ranking-fair.

\begin{theorem}
    The canonical pattern set allows a ranking-fair schedule for~$n=8$.
\end{theorem}
\begin{proof}
    A corresponding schedule is given in Table~\ref{tab:cps:n8:rf}.
    It is straightforward to verify that teams $4$ and $5$, $1$ and $8$, $2$ and $7$, $3$ and $6$ have breaks in rounds $1$, $3$, $5$, and $7$, respectively, following the canonical pattern set.
    Furthermore, the schedule follows the ranking HAPs as in Table~\ref{tab:rankingHAPset}.
\end{proof}

\begin{table}
    \centering
    \caption{A ranking-fair schedule for $n=8$ teams using the canonical pattern set. The team playing at home is displayed first.}\label{tab:cps:n8:rf}
    \begin{tabular}{r|cccc}
        \toprule
        Round & \multicolumn{4}{l}{Matches} \\
        \midrule
        $1$ & $5$--$1$ & $2$--$3$ & $8$--$4$ & $6$--$7$ \\
        $2$ & $1$--$6$ & $4$--$2$ & $3$--$8$ & $7$--$5$ \\
        $3$ & $1$--$8$ & $2$--$7$ & $5$--$3$ & $6$--$4$ \\
        $4$ & $7$--$1$ & $8$--$2$ & $3$--$6$ & $4$--$5$ \\
        $5$ & $1$--$4$ & $6$--$2$ & $7$--$3$ & $5$--$8$ \\
        $6$ & $3$--$1$ & $2$--$5$ & $4$--$7$ & $8$--$6$ \\
        $7$ & $1$--$2$ & $3$--$4$ & $5$--$6$ & $7$--$8$ \\ \bottomrule
    \end{tabular}
\end{table}

For~$n = 10$, the canonical pattern set does not allow a ranking-fair schedule, as verified by the IP model in Section~\ref{sec:specialcasen}.
In fact, for any~$n \ge 12$, the CPS does not allow a ranking-fair schedule, as shown in the following result.

\begin{theorem} \label{thm:cps}
    The canonical pattern set does not allow a ranking-fair schedule for~$n \ge 12$.
\end{theorem}
\begin{proof}
    Suppose for the sake of contradiction that there exists a schedule using the canonical pattern set that is ranking-fair.
    Let the breaks of the schedule occur in rounds~$1, 2, 4, 6, \dots, n - 2$.
    The strongest team, team~$1$, is assigned a HAP from the canonical pattern set;
    say that the break for this team occurs in round~$r$.
    It must be the case that there are also breaks in either rounds~$r+2, r+4, r+6$, or~$r-2, r-4, r-6$.
    Without loss of generality, we assume the former, as otherwise we can modify the schedule such that the rounds are played in reverse order.
    Consider the HAPs that have home breaks in rounds~$r, r+2, r+4, r+6$.
    Because these patterns have one more home match than away matches, they are assigned to odd teams.
    We call these teams~$1$, $x$, $y$, and $z$, respectively.
    Note that outside of the rounds~$r$ up to round~$r+5$, these patterns are equal.
    Therefore, matches between the teams assigned to these patterns must occur in one of the six rounds in which the patterns are different.
    See also Table~\ref{tab:cpsproof}.

    Note that in a ranking-fair schedule, team~$1$ plays away against all odd opponents.
    Therefore, the only round in which the match between team~$1$ and~$x$ can occur, is in round~$r+1$.
    This forces the match between $1$ and $y$ in round~$r+3$, and consequently the match between $1$ and $z$ in round $r+5$.
    Similarly, the match between $y$ and $z$ must occur in round~$r+4$ and between $x$ and $y$ in $r+2$.
    Because of ranking fairness and since $z$ plays away against $y$ and $y$ away against $x$, we now know that~$z < y$ and $y < x$.
    The only match remaining is between $x$ and $z$, which can only be scheduled in round $r+3$, with $x$ playing away.
    Therefore, $z > x$, which is a contradiction with~$z < y < x$.
    This concludes the proof.
\end{proof}

\begin{table}
    \def\home{H}
    \def\away{A}
    \centering
    \caption{A restricted view of a CPS schedule.}\label{tab:cpsproof}
    \begin{tabular}{l|cccccccc}
        \toprule
        Team & $r-1$ & $r$ & $r+1$ & $r+2$ & $r+3$ & $r+4$ & $r+5$ & $r+6$ \\ \midrule
        $1$ & {\home} & {\home} & {\away} & \home & {\away} & \home & {\away} & \home \\
        $x$ & \home & \away & \home & {\home} & {\away} & \home & \away & \home \\
        $y$ & \home & \away & \home & {\away} & {{\home}} & {{\home}} & \away & \home \\
        $z$ & \home & \away & \home & \away & {\home} & {\away} & {{\home}} & {\home} \\ \bottomrule
    \end{tabular}
\end{table}

This argument is also applicable for other HAP sets where the corresponding D-sequence contains the sub-sequence~$2222$.

\begin{lemma}
    A ranking-fair D-sequence cannot contain the sub-sequence~$2222$.
\end{lemma}
\begin{proof}
    Consider a schedule with a D-sequence that contains the sub-sequence~$2222$, and say that breaks occur in rounds~$r$, $r+2$, $r+4$, $r+6$, and $r+6$.
    Let $a, b, c, d, e$ denote the teams that have a home-break in these rounds, respectively.
    Using the proof of Theorem~\ref{thm:cps}, team $a$ and $b$ cannot be the strongest team, nor the weakest team.
    By symmetry of the schedule, the same holds for team~$d$ and $e$, that can neither be the strongest nor the weakest.
    Hence, team~$c$ needs to be \emph{both} the strongest and the weakest team, a contradiction.
\end{proof}

\section{Examples from Practice}
\label{sec:practice}
There are many instances of tournaments that are organized as an SRR that feature an asymmetry.
Here we discuss the following examples: the Tata Steel Chess tournament, the Danish Soccer League when it was a triple round robin, and the Dutch baseball league.
The code and data used for this section can be found in our repository, available at \url{https://github.com/stenwessel/fair-schedules-srr-ranked-participants}.

We first highlight the 2002 edition of the Tata Steel Chess tournament, as mentioned in the prologue.
Table~\ref{tab:tata2002view2} shows in the rows the players of the tournament, ordered by their Elo rating.
The columns are the opponents for every player, also ordered by their strength, where every cell indicates whether the player plays its opponent with white or black.
It is clearly visible that for many players there is significant imbalance over the ranked opponents.
Figure~\ref{tab:tata-fscores} shows for several years of the tournament the value~$F$ of the schedule, which represents the amount of imbalance;
we give its definition and motivation in Section~\ref{sec:probdes}.

\begin{table}
    \centering
    \caption{Matches played with white and black for every player, ordered by the rank of their opponent, for the 2002 edition of the Tata Stell Chess tournament.}\label{tab:tata2002view2}
    \begin{tabular}{rllccccccccccccc}
        \toprule
    Rank & Player & Elo & \multicolumn{13}{l}{Opponents sorted by rank $\longrightarrow$} \\\midrule
    1. & Morozevich, A. & 2742 & \whitesq & \blacksq & \blacksq & \blacksq & \whitesq & \whitesq & \whitesq & \whitesq & \blacksq & \blacksq & \whitesq & \blacksq & \blacksq \\
    2. & Adams, M. & 2742 & \blacksq & \blacksq & \blacksq & \blacksq & \blacksq & \blacksq & \blacksq & \whitesq & \whitesq & \whitesq & \whitesq & \whitesq & \whitesq \\
    3. & Leko, P. & 2713 & \whitesq & \whitesq & \blacksq & \blacksq & \whitesq & \blacksq & \blacksq & \whitesq & \whitesq & \whitesq & \blacksq & \blacksq & \whitesq \\
    4. & Bareev, E. & 2707 & \whitesq & \whitesq & \whitesq & \blacksq & \whitesq & \blacksq & \blacksq & \whitesq & \whitesq & \whitesq & \blacksq & \blacksq & \blacksq \\
    5. & Gelfand, B. & 2703 & \whitesq & \whitesq & \whitesq & \whitesq & \whitesq & \blacksq & \blacksq & \whitesq & \whitesq & \blacksq & \blacksq & \blacksq & \blacksq \\
    6. & Van Wely, L. & 2697 & \blacksq & \whitesq & \blacksq & \blacksq & \blacksq & \blacksq & \blacksq & \whitesq & \whitesq & \whitesq & \blacksq & \whitesq & \whitesq \\
    7. & Kasimdzhanov, R. & 2695 & \blacksq & \whitesq & \whitesq & \whitesq & \whitesq & \whitesq & \whitesq & \blacksq & \blacksq & \blacksq & \blacksq & \blacksq & \blacksq \\
    8. & Khalifman, A. & 2688 & \blacksq & \whitesq & \whitesq & \whitesq & \whitesq & \whitesq & \blacksq & \whitesq & \blacksq & \blacksq & \blacksq & \blacksq & \blacksq \\
    9. & Lautier, J. & 2687 & \blacksq & \blacksq & \blacksq & \blacksq & \blacksq & \blacksq & \whitesq & \blacksq & \whitesq & \whitesq & \whitesq & \whitesq & \whitesq \\
    10. & Dreev, A. & 2683 & \whitesq & \blacksq & \blacksq & \blacksq & \blacksq & \blacksq & \whitesq & \whitesq & \blacksq & \whitesq & \whitesq & \whitesq & \whitesq \\
    11. & Grischuk, A. & 2671 & \whitesq & \blacksq & \blacksq & \blacksq & \whitesq & \blacksq & \whitesq & \whitesq & \blacksq & \blacksq & \whitesq & \whitesq & \whitesq \\
    12. & Piket, J. & 2659 & \blacksq & \blacksq & \whitesq & \whitesq & \whitesq & \whitesq & \whitesq & \whitesq & \blacksq & \blacksq & \blacksq & \blacksq & \blacksq \\
    13. & Gurevich, M. & 2641 & \whitesq & \blacksq & \whitesq & \whitesq & \whitesq & \blacksq & \whitesq & \whitesq & \blacksq & \blacksq & \blacksq & \whitesq & \blacksq \\
    14. & Timman, J. & 2605 & \whitesq & \blacksq & \blacksq & \whitesq & \whitesq & \blacksq & \whitesq & \whitesq & \blacksq & \blacksq & \blacksq & \whitesq & \whitesq \\ \bottomrule
    \end{tabular}
\end{table}

\begin{figure}
    \centering
    \includegraphics[width=.7\textwidth]{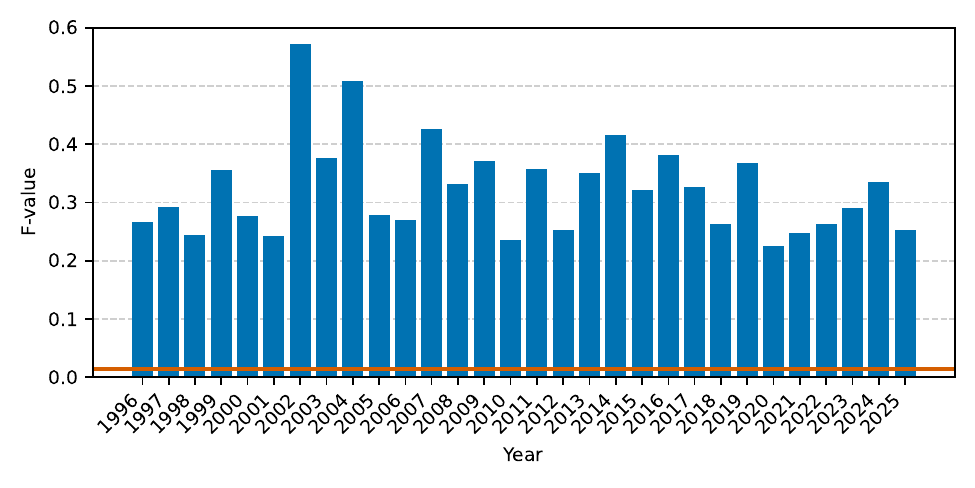}
    \caption{The imbalance value~$F$ for the schedule of several editions of the Tata Steel Chess tournament in which the tournament was organized with 14 players. The horizontal line indicates the optimal $F$-value for 14 players.}\label{tab:tata-fscores}
\end{figure}

The Danish Superliga soccer league was organized between seasons 1995--1996 and 2015--2016 as a triple round-robin tournament with 12~participating teams.
Every team plays every other team either once or twice at home.
We consider the team that plays the other team twice at home as having the home field advantage between the two teams.
Table~\ref{tab:dansocc2008} shows the home advantage for each team, with the opponents sorted by rank.
The ranking of the teams is based on the results of the previous season.
The imbalance for the 2008--2009 season equals~$F = 0.476$.

\begin{table}
    \centering
    \caption{Home field advantage (white square) for every team, ordered by the rank of their opponent, for the 2008--2009 season of the Danish Superliga.}\label{tab:dansocc2008}
    \begin{tabular}{rlccccccccccc}
        \toprule
    Rank & Team & \multicolumn{11}{l}{Opponents sorted by rank $\longrightarrow$} \\\midrule
        1. & Aalborg BK & \whitesq & \blacksq & \whitesq & \whitesq & \blacksq & \whitesq & \blacksq & \whitesq & \whitesq & \blacksq & \blacksq \\
        2. & FC Midtjylland & \blacksq & \blacksq & \whitesq & \whitesq & \blacksq & \whitesq & \blacksq & \whitesq & \whitesq & \whitesq & \blacksq \\
        3. & FC København & \whitesq & \whitesq & \whitesq & \blacksq & \blacksq & \whitesq & \blacksq & \whitesq & \whitesq & \blacksq & \blacksq \\
        4. & Odense BK & \blacksq & \blacksq & \blacksq & \blacksq & \whitesq & \whitesq & \whitesq & \whitesq & \whitesq & \blacksq & \whitesq \\
        5. & AC Horsens & \blacksq & \blacksq & \whitesq & \whitesq & \whitesq & \blacksq & \whitesq & \blacksq & \blacksq & \whitesq & \whitesq \\
        6. & Randers FC & \whitesq & \whitesq & \whitesq & \blacksq & \blacksq & \whitesq & \blacksq & \blacksq & \whitesq & \whitesq & \blacksq \\
        7. & Esbjerg fB & \blacksq & \blacksq & \blacksq & \blacksq & \whitesq & \blacksq & \blacksq & \whitesq & \whitesq & \whitesq & \whitesq \\
        8. & Brøndby IF & \whitesq & \whitesq & \whitesq & \blacksq & \blacksq & \whitesq & \whitesq & \blacksq & \blacksq & \blacksq & \blacksq \\
        9. & FC Nordsjælland & \blacksq & \blacksq & \blacksq & \blacksq & \whitesq & \whitesq & \blacksq & \whitesq & \blacksq & \whitesq & \whitesq \\
        10. & Aarhus GF & \blacksq & \blacksq & \blacksq & \blacksq & \whitesq & \blacksq & \blacksq & \whitesq & \whitesq & \whitesq & \whitesq \\
        11. & Vejle BK & \whitesq & \blacksq & \whitesq & \whitesq & \blacksq & \blacksq & \blacksq & \whitesq & \blacksq & \blacksq & \whitesq \\
        12. & Sønderjyske & \whitesq & \whitesq & \whitesq & \blacksq & \blacksq & \whitesq & \blacksq & \whitesq & \blacksq & \blacksq & \blacksq \\\bottomrule
    \end{tabular}
\end{table}

The Dutch baseball league has a regular season with nine teams organized as a triple round robin, where every team plays every other team either once or twice at home.
For the 2024 season, Table~\ref{tab:hb2024} shows the home advantage for each team, with the opponents sorted by rank.
The ranking is based on the results of the previous season.
For this schedule, $F = 0.497$.

\begin{table}
    \centering
    \caption{Home field advantage (white square) for every team, ordered by the rank of their opponent, for the 2024 season of the Dutch baseball league.}\label{tab:hb2024}
    \begin{tabular}{rlcccccccc}
        \toprule
        Rank & Team & \multicolumn{8}{l}{Opponents sorted by rank $\longrightarrow$} \\\midrule
        1. & Amsterdam Pirates & \blacksq & \whitesq & \whitesq & \blacksq & \whitesq & \blacksq & \whitesq & \blacksq \\
        2. & Curaçao Neptunus & \whitesq & \blacksq & \blacksq & \blacksq & \blacksq & \whitesq & \whitesq & \whitesq \\
        3. & HCAW & \blacksq & \whitesq & \whitesq & \whitesq & \whitesq & \blacksq & \blacksq & \blacksq \\
        4. & RCH-Pinguïns & \blacksq & \whitesq & \blacksq & \whitesq & \whitesq & \whitesq & \blacksq & \blacksq \\
        5. & Oosterhout Twins & \whitesq & \whitesq & \blacksq & \blacksq & \blacksq & \whitesq & \blacksq & \whitesq \\
        6. & DSS/Kinheim & \blacksq & \whitesq & \blacksq & \blacksq & \whitesq & \whitesq & \blacksq & \whitesq \\
        7. & Hoofddorp Pioniers & \whitesq & \blacksq & \whitesq & \blacksq & \blacksq & \blacksq & \whitesq & \whitesq \\
        8. & Quick Amersfoort & \blacksq & \blacksq & \whitesq & \whitesq & \whitesq & \whitesq & \blacksq & \blacksq \\
        9. & UVV & \whitesq & \blacksq & \whitesq & \whitesq & \blacksq & \blacksq & \blacksq & \whitesq \\\bottomrule
    \end{tabular}
\end{table}

\subsection*{Acknowledgements}
This research is supported by NWO Gravitation Project NETWORKS, Grant Number 024.002.003.

\end{document}